\documentclass[10pt]{amsart}
\usepackage{amsmath,amscd}
\usepackage{amsbsy}
\usepackage{amssymb}
\usepackage{amscd,amsthm}
\usepackage[all,cmtip]{xy}

\newtheorem{thm}{Theorem}
\numberwithin{thm}{section}
\newtheorem{lem}[thm]{Lemma}

\newtheorem{cor}[thm]{Corollary}
\newtheorem{exam}[thm]{Example}
\newtheorem{rema}[thm]{Remark}

\newtheorem{defi}[thm]{Definition}

\newtheorem*{que2}{Question}

\begin{document}
\begin{center}
\huge{On non-existence of full exceptional collections on some relative flags}\\[1cm]
\end{center}

\begin{center}

\large{Sa$\mathrm{\check{s}}$a Novakovi$\mathrm{\acute{c}}$}\\[0,4cm]
{\small December 2019}\\[0,3cm]
\end{center}

\noindent{\small \textbf{Abstract}. 
In this paper we show that certain relative flags cannot have full exceptional collections. We also prove that some of these flags are categorical representable in dimension zero if and only if they admit a full exceptional collection. As a consequence, these flags are categorical representable in dimension zero if and only if they have $k$-rational points if and only if they are $k$-rational.	Moreover, we calculate the categorical representability dimension for the flags under consideration.\\


	\section{Introduction}
In \cite{NO2Z} the author proved among others that non-trivial twisted flags of type $A_n$ or $C_n$ cannot have full exceptional collections. It is therefore natural to ask whether twisted forms of relative flags or relative flags over bases (which do not have full exceptional collections) can have full exceptional collections. In the last years it became clear that semiorthogonal decompositions detect birational properties of a given variety (see for instance \cite{ABZ} and references therein). In this context, we follow the ideas presented in \cite{NO2Z} and relate the existence of full exceptional collections on certain relative flags $X$ to the rationality of $X$. The first result of the present paper is the following theorem which is certainly known to the experts but, to our best knowledge, stated nowhere. 

\begin{thm}
	Let $Z$ be a non-trivial twisted flag of type $A_n$ or $C_n$ over a field $k$ and $\pi\colon X\rightarrow Z$ a flat and proper morphism with $X$ a smooth projective variety over $k$. Assume there is a semiorthogonal decomposition $D^b(X)=\langle \pi^*D^b(Z)\otimes \mathcal{E}_1,...,\pi^*D^b(Z)\otimes \mathcal{E}_r\rangle$ with $D^b(Z)$ being equivalent to $\pi^*D^b(Z)\otimes \mathcal{E}_i$ via $\pi^*(-)\otimes \mathcal{E}_i$. Then $X$ cannot have a full exceptional collection.
\end{thm}
Examples are given by Corollaries 5.1 and 5.2. Note that although the variety $X$ from Theorem 1.1 cannot have a full exceptional collection, it can have a tilting bundle (see \cite{NO3Z} for examples). 
Closely related to the problem of the existence of full exceptional collections is the question whether the considered variety admits a $k$-rational point or is $k$-rational. A potential measure for rationality was introduced by Bernardara and Bolognesi \cite{BBZ} with the notion of categorical representability. We use the definition given in \cite{AB1Z}. A $k$-linear triangulated category $\mathcal{T}$ is said to be \emph{representable in dimension $m$} if there is a semiorthogonal decomposition (see Section 3 for the definition) $\mathcal{T}=\langle \mathcal{A}_1,...,\mathcal{A}_n\rangle$ and for each $i=1,...,n$ there exists a smooth projective connected variety $Y_i$ with $\mathrm{dim}(Y_i)\leq m$, such that $\mathcal{A}_i$ is equivalent to an admissible subcategory of $D^b(Y_i)$. We use the following notation
\begin{eqnarray*}
	\mathrm{rdim}(\mathcal{T}):=\mathrm{min}\{m\mid \mathcal{T}\  \textnormal{is representable in dimension m}\},
\end{eqnarray*}
whenever such a finite $m$ exists. Let $X$ be a smooth projective $k$-variety. One says $X$ is \emph{representable in dimension} $m$ if $D^b(X)$ is representable in dimension $m$. We will use the following notations:
\begin{eqnarray*}
	\mathrm{rdim}(X):=\mathrm{rdim}(D^b(X)),\thickspace \mathrm{rcodim}(X):= \mathrm{dim}(X)-\mathrm{rdim}(X).
\end{eqnarray*} 
We recall the following question which was formulated in \cite{MBTZ}:
\begin{que2}
	Let $X$ be a smooth projective variety over $k$ of dimension at least 2. Suppose $X$ is $k$-rational. Do we have $\mathrm{rcodim}(X)\geq 2$ ?
\end{que2} 
There are several results suggesting that this question has a positive answer, see \cite{MBTZ}, \cite{ABZ}, \cite{BBZ}, \cite{NO2Z} and references therein.
Now let $G=\mathrm{PSO}_n$ be over $k$ with $n$ even and $\mathrm{char}(k)\neq 2$. Given a 1-cocycle $\gamma\colon \mathrm{Gal}(k^s|k)\rightarrow \mathrm{PSO}_n(k^s)$ we get a twisted form of a quadric ${_\gamma}Q$ and a central simple $k$-algebra $(A,\sigma)$ of degree $n$ with involution associated to $\gamma$ (see \cite{KNUZ}). Note that ${_\gamma}Q$ is isomorphic to the involution variety $\mathrm{IV}(A,\sigma)$ of Section 2. We assume that the degree of $A$ is even. This is no restriction (see Section 2 for an explanation). For any splitting field $L$ of $A$, the variety ${_\gamma}Q\otimes_k L$ is isomorphic to a smooth quadric in $\mathbb{P}^{n-1}_L$. 
Note that the (generalized) Brauer--Severi varieties are obtained as quotients of $G=\mathrm{PGL}_n$ by a certain parabolic subgroup $P$ and by twisting with a 1-cocycle $\gamma\colon\mathrm{Gal}(k^s|k)\rightarrow \mathrm{PGL}_n(k^s)$.
In Section 6 we show the following:
\begin{thm}
	Let $Z$ be one of the following varieties:
	\begin{itemize}
		\item[(\textbf{i})]	a Brauer--Severi variety over an arbitrary field $k$.
		\item[(\textbf{ii})] a generalized Brauer--Severi variety over a field $k$ of characteristic zero.
		\item[(\textbf{iii})] a twisted quadric associated to a central simple algebra $(A,\sigma)$ with involution of orthogonal type having trivial discriminant $\delta(A,\sigma)$. 
		\item[(\textbf{iv})] a twisted quadric associated to a central simple $\mathbb{R}$-algebra $(A,\sigma)$ with involution of orthogonal type.
	\end{itemize} 
	Let $\pi\colon X\rightarrow Z$ be a flat and proper morphism where $X$ is a smooth projective variety over $k$. Assume there is a semiorthogonal decomposition $D^b(X)=\langle \pi^*D^b(Z)\otimes \mathcal{E}_1,...,\pi^*D^b(Z)\otimes \mathcal{E}_r\rangle$ as in Theorem 1.1. Then $\mathrm{rdim}(X)=0$ if and only if the 1-cocycle defining $Z$ is trivial.
\end{thm}

\begin{cor}
	Under the assumptions of Theorem 1.2, $\mathrm{rdim}(X)=0$ if and only if $X$ admits a full exceptional collection.
\end{cor}
\begin{cor}
	Let $Z$ be a Brauer--Severi variety over an arbitrary field $k$ or a smooth twisted quadric as in Theorem 1.2 and $\mathcal{E}$ a vector bundle on $Z$. Let $X$ be $\mathbb{P}_Z(\mathcal{E})$ or $\mathrm{Grass}_Z(d,\mathcal{E})$. Then $\mathrm{rdim}(X)=0$ if and only if $X$ admits a $k$-rational point if and only if $X$ is $k$-rational. In particular, if $X$ is $k$-rational, then $\mathrm{rcodim}(X)\geq 2$ 
\end{cor}
\noindent
Notice that Corollary 1.4 is not true if $Z$ is a generalized Brauer--Severi variety. Indeed, as \cite{NO2Z}, Remark 6.7 shows, there are non-trivial generalized Brauer--Severi varieties $Z$ admitting a rational point although $\mathrm{rdim}(Z)=1$.

The number $\mathrm{rdim}$ is closely related to the \emph{motivic categorical dimension} which is introduced in \cite{MBZ} for arbitrary fields. In $\mathrm{char}(k)=0$, motivic categorical dimension of a smooth projective variety $X$ is defined to be the smallest $d$ such that $\mu([X])$ lies in $PT_d(k)$. Here $PT(k)$ denotes the Grothendieck ring of dg categories (see \cite{BLLZ}), $PT_d(k)\subset PT(k)$ the additive subgroup generated by the smallest saturated monoid containing classes of pretriangulated dg categories of categorical dimension at least $d$ and $\mu\colon K_0(Var(k))\rightarrow PT(k)$ the motivic measure defined by Bondal, Larsen and Lunts in \cite{BLLZ}. If one denotes the motivic categorical dimension of $D^b(X)$ by $\mathrm{mcd}(X)$, one has $\mathrm{mcd}(X)\leq \mathrm{rdim}(X)$ (see \cite{MBZ}). And indeed, results from \cite{MBZ} suggest that motivic categorical dimension can be used to define a birational invariant. A natural problem is then to calculate the numbers $\mathrm{mcd}(X)$  and $\mathrm{rdim}(X)$ and find conditions for which $X$ equality $\mathrm{mcd}(X)=\mathrm{rdim}(X)$ holds. In this context we show the following result.
\begin{thm}
	Let $Z$ be one of the following varieties:
	\begin{itemize}
		\item[(\textbf{i})]	a Brauer--Severi variety over an arbitrary field $k$.
		\item[(\textbf{ii})] a generalized Brauer--Severi variety over a field $k$ of characteristic zero.
		\item[(\textbf{iii})] a twisted quadric associated to a central simple algebra $(A,\sigma)$ with involution of orthogonal type having trivial discriminant $\delta(A,\sigma)$. 
		\item[(\textbf{iv})] a twisted quadric associated to a central simple $\mathbb{R}$-algebra $(A,\sigma)$ with involution of orthogonal type.
	\end{itemize} 
	Let $X$ be $\mathbb{P}_Z(\mathcal{E})$ or $\mathrm{Grass}_Z(d,\mathcal{E})$ and denote by $A$ the central simple algebra associated to $Z$. Then $\mathrm{rdim}(X)\leq \mathrm{ind}(A)-1$. Moreover, if $\mathrm{ind}(A)\leq 3$, then $\mathrm{rdim}(X)=\mathrm{ind}(A)-1$. In particular $\mathrm{mcd}(X)\leq \mathrm{ind}(A)-1$.
\end{thm}

\section{Generalized Brauer--Severi varieties and twisted quadrics}
A Brauer--Severi variety of dimension $n$ is a scheme $X$ of finite type over $k$ such that $X\otimes_k L\simeq \mathbb{P}^n$ for a finite field extension $k\subset L$. A field extension $k\subset L$ for which $X\otimes_k L\simeq \mathbb{P}^n$ is called \emph{splitting field} of $X$. Clearly, $k^s$ and $\bar{k}$ are splitting fields for any Brauer--Severi variety. In fact, every Brauer--Severi variety always splits over a finite Galois extension of $k$. It follows from descent theory that $X$ is projective, integral and smooth over $k$. Via non-commutative Galois cohomology, Brauer--Severi varieties of dimension $n$ are in one-to-one correspondence with central simple algebras $A$ of degree $n+1$. For details and proofs on all mentioned facts we refer to \cite{ARZ} and \cite{GSZ}.

To a central simple $k$-algebra $A$ one can also associate twisted forms of Grassmannians. Let $A$ be of degree $n$ and $1\leq d\leq n$. Consider the subset of $\mathrm{Grass}_k(d\cdot n, A)$ consisting of those subspaces of $A$ that are left ideals $I$ of dimension $d\cdot n$. This subset can be given the structure of a projective variety which turns out to be a generalized Brauer--Severi variety. It is denoted by $\mathrm{BS}(d,A)$. After base change to some splitting field $E$ of $A$ the variety $\mathrm{BS}(d,A)$ becomes isomorphic to $\mathrm{Grass}_E(d,n)$. If $d=1$ the generalized Brauer--Severi variety is the Brauer--Severi variety associated to $A$. Note that $\mathrm{BS}(d,A)$ is a Fano variety. For details see \cite{BLZ}. 

Finally, to a central simple algebra $A$ of degree $n$ with involution $\sigma$ of the first kind over a field $k$ of $\mathrm{char}(k)\neq 2$ one can associate the \emph{involution variety} $\mathrm{IV}(A,\sigma)$. This variety can be described as the variety of $n$-dimensional right ideals $I$ of $A$ such that $\sigma(I)\cdot I=0$. If $A$ is split so $(A,\sigma)\simeq (M_n(k), q^*)$, where $q^*$ is the adjoint involution defined by a quadratic form $q$ one has $\mathrm{IV}(A,\sigma)\simeq V(q)\subset \mathbb{P}^{n-1}_k$. Here $V(q)$ is the quadric determined by $q$. By construction such an involution variety  $\mathrm{IV}(A,\sigma)$ becomes a quadric in $\mathbb{P}^{n-1}_L$ after base change to some splitting field $L$ of $A$. In this way the involution variety is a twisted form of a smooth quadric as defined before. Recall from \cite{DTZ} that a splitting field $L$ of $A$ is called \emph{isotropically} if $(A,\sigma)\otimes_k L\simeq (M_n(L), q^*)$ with $q$ an isotropic quadratic form over $L$. Although the degree of $A$ is arbitrary, (when $\mathrm{char}(k)\neq 2$), the case where degree of $A$ is odd does not give anything new, since central simple algebras of odd degree with involution of the first kind are split (see \cite{KNUZ}, Corollary 2.8). For details on the construction and further properties on involution varieties and the corresponding algebras we refer to \cite{DTZ}.

\section{Exceptional collections and semiorthogonal decompositions} 
Let $\mathcal{D}$ be a triangulated category and $\mathcal{C}$ a triangulated subcategory. The subcategory $\mathcal{C}$ is called \emph{thick} if it is closed under isomorphisms and direct summands. For a subset $A$ of objects of $\mathcal{D}$ we denote by $\langle A\rangle$ the smallest full thick subcategory of $\mathcal{D}$ containing the elements of $A$. 
Furthermore, we define $A^{\perp}$ to be the subcategory of $\mathcal{D}$ consisting of all objects $M$ such that $\mathrm{Hom}_{\mathcal{D}}(E[i],M)=0$ for all $i\in \mathbb{Z}$ and all elements $E$ of $A$. We say that $A$ \emph{generates} $\mathcal{D}$ if $A^{\perp}=0$. Now assume $\mathcal{D}$ admits arbitrary direct sums. An object $B$ is called \emph{compact} if $\mathrm{Hom}_{\mathcal{D}}(B,-)$ commutes with direct sums. Denoting by $\mathcal{D}^c$ the subcategory of compact objects we say that $\mathcal{D}$ is \emph{compactly generated} if the objects of $\mathcal{D}^c$ generate $\mathcal{D}$. One has the following important theorem (see \cite{BVZ}, Theorem 2.1.2).
\begin{thm}
	Let $\mathcal{D}$ be a compactly generated triangulated category. Then a set of objects $A\subset \mathcal{D}^c$ generates $\mathcal{D}$ if and only if $\langle A\rangle=\mathcal{D}^c$.  
\end{thm}
For a smooth projective scheme $X$ over $k$, we denote by $D(\mathrm{Qcoh}(X))$ the derived category of quasicoherent sheaves on $X$. The bounded derived category of coherent sheaves is denoted by $D^b(X)$. Note that $D(\mathrm{Qcoh}(X))$ is compactly generated with compact objects being all of $D^b(X)$. For details on generating see \cite{BVZ}.

\begin{defi}
	\textnormal{Let $A$ be a division algebra over $k$, not necessarily central. An object $\mathcal{E}\in D^b(X)$ is called \emph{w-exceptional} if $\mathrm{End}(\mathcal{E})=A$ and $\mathrm{Hom}(\mathcal{E},\mathcal{E}[r])=0$ for $r\neq 0$. If $A=k$ the object is called \emph{exceptional}. If $A$ is a separable $k$-algebra, the object $\mathcal{E}$ is called \emph{separable-exceptional}. } 
\end{defi}
\begin{defi}
	\textnormal{A totally ordered set $\{\mathcal{E}_1,...,\mathcal{E}_n\}$ of w-exceptional (resp. separable-exceptional) objects on $X$ is called an \emph{w-exceptional collection} (resp. \emph{separable-exceptional collection}) if $\mathrm{Hom}(\mathcal{E}_i,\mathcal{E}_j[r])=0$ for all integers $r$ whenever $i>j$. An w-exceptional (resp. separable-exceptional) collection is \emph{full} if $\langle\{\mathcal{E}_1,...,\mathcal{E}_n\}\rangle=D^b(X)$ and \emph{strong} if $\mathrm{Hom}(\mathcal{E}_i,\mathcal{E}_j[r])=0$ whenever $r\neq 0$. If the set $\{\mathcal{E}_1,...,\mathcal{E}_n\}$ consists of exceptional objects it is called \emph{exceptional collection}.}
\end{defi}
Notice that the direct sum of objects forming a full strong w-exceptional (resp. separable-exceptional) collection is a tilting object. 
\begin{exam}
	\textnormal{Let $\mathbb{P}^n$ be the projective space and consider the ordered collection of invertible sheaves $\{\mathcal{O}_{\mathbb{P}^n}, \mathcal{O}_{\mathbb{P}^n}(1),...,\mathcal{O}_{\mathbb{P}^n}(n)\}$. In \cite{BEZ} Beilinson showed that this is a full strong exceptional collection. }
\end{exam}
\begin{exam}
	\textnormal{Let $X=\mathbb{P}^1\times\mathbb{P}^1$. Then $\{\mathcal{O}_X,\mathcal{O}_X(1,0), \mathcal{O}_X(0,1), \mathcal{O}_X(1,1)\}$ is a full strong exceptional collection on $X$. We use the notion $\mathcal{O}_X(i,j)$ for $\mathcal{O}(i)\boxtimes\mathcal{O}(j)$.}
\end{exam}
A generalization of the notion of a full w-exceptional collection is that of a semiorthogonal decomposition of $D^b(X)$. Recall that a full triangulated subcategory $\mathcal{D}$ of $D^b(X)$ is called \emph{admissible} if the inclusion $\mathcal{D}\hookrightarrow D^b(X)$ has a left and right adjoint functor. 
\begin{defi}
	\textnormal{Let $X$ be a smooth projective variety over $k$. A sequence $\mathcal{D}_1,...,\mathcal{D}_n$ of full triangulated subcategories of $D^b(X)$ is called \emph{semiorthogonal} if all $\mathcal{D}_i\subset D^b(X)$ are admissible and $\mathcal{D}_j\subset \mathcal{D}_i^{\perp}=\{\mathcal{F}\in D^b(X)\mid \mathrm{Hom}(\mathcal{G},\mathcal{F})=0$, $\forall$ $ \mathcal{G}\in\mathcal{D}_i\}$ for $i>j$. Such a sequence defines a \emph{semiorthogonal decomposition} of $D^b(X)$ if the smallest full thick subcategory containing all $\mathcal{D}_i$ equals $D^b(X)$.}
\end{defi}
\noindent
For a semiorthogonal decomposition we write $D^b(X)=\langle \mathcal{D}_1,...,\mathcal{D}_n\rangle$.
\begin{exam}
	\textnormal{Let $\mathcal{E}_1,...,\mathcal{E}_n$ be a full w-exceptional collection on $X$. It is easy to verify that by setting $\mathcal{D}_i=\langle\mathcal{E}_i\rangle$ one gets a semiorthogonal decomposition $D^b(X)=\langle \mathcal{D}_1,...,\mathcal{D}_n\rangle$.}
\end{exam}
\noindent
For a wonderful and comprehensive overview of the theory on semiorthogonal decompositions and its relevance in algebraic geometry we refer to \cite{KUZ}. 



\section{Recollections on noncommutative motives}
We refer to the book \cite{GTAZ} (alternatively see \cite{TTZ} and \cite{MTZ} for a survey on noncommutative motives). Let $\mathcal{A}$ be a small dg category. The homotopy category $H^0(\mathcal{A})$ has the same objects as $\mathcal{A}$ and as morphisms $H^0(\mathrm{Hom}_{\mathcal{A}}(x,y))$. A source of examples is provided by schemes since the derived category of perfect complexes $\mathrm{perf}(X)$ of any quasi-compact quasi-seperated scheme $X$ admits a canonical dg enhancement $\mathrm{perf}_{dg}(X)$ (for details see \cite{KELZ}). Note that for smooth projective $k$-schemes $X$ one has $D^b(X)=\mathrm{perf}(X)$. Denote by $\textbf{dgcat}$ the category of small dg categories. The \emph{opposite} dg category $\mathcal{A}^{op}$ has the same objects as $\mathcal{A}$ and $\mathrm{Hom}_{\mathcal{A}^{op}}(x,y):=\mathrm{Hom}_{\mathcal{A}}(y,x)$. A \emph{right $\mathcal{A}$-module} is a dg functor $\mathcal{A}^{op}\rightarrow C_{dg}(k)$ with values in the dg category $C_{dg}(k)$ of complexes of $k$-vector spaces. We write $C(\mathcal{A})$ for the category of right $\mathcal{A}$-modules. Recall form \cite{KELZ} that the \emph{derived category} $D(\mathcal{A})$ of $\mathcal{A}$ is the localization of $C(\mathcal{A})$  with respect to quasi-isomorphisms. A dg functor $F\colon \mathcal{A}\rightarrow \mathcal{B}$ is called \emph{derived Morita equivalence} if the restriction of scalars functor $D(\mathcal{B})\rightarrow D(\mathcal{A})$ is an equivalence. The \emph{tensor product} $\mathcal{A}\otimes \mathcal{B}$ of two dg categories is defined as follows: the set of objects is the cartesian product of the sets of objects in $\mathcal{A}$ and $\mathcal{B}$ and $\mathrm{Hom}_{\mathcal{A}\otimes \mathcal{B}}((x,w),(y,z)):=\mathrm{Hom}_{\mathcal{A}}(x,y)\otimes\mathrm{Hom}_{\mathcal{B}}(w,z)$ (see \cite{KELZ}). Given two dg categories $\mathcal{A}$ and $\mathcal{B}$, let $\mathrm{rep}(\mathcal{A},\mathcal{B})$ be the full triangulated subcategory of $D(\mathcal{A}^{op}\otimes \mathcal{B})$ consisting of those $\mathcal{A}-\mathcal{B}$-bimodules $M$ such that $M(x,-)$ is a compact object of $D(\mathcal{B})$ for every object $x\in \mathcal{A}$. 
The category $\textbf{dgcat}$ of all (small) dg categories and dg functors carries a Quillen model structure whose weak equivalences are Morita equivalences. Let us denote by $\mathrm{Hmo}$ the homotopy category hence obtained and by $\mathrm{Hmo}_0$ its additivization. Now to any small dg category $\mathcal{A}$ one can associate functorially its \emph{noncommutative motive} $U(\mathcal{A})$ which takes values in $\mathrm{Hmo}_0$. This functor $U\colon \textbf{dgcat}\rightarrow \mathrm{Hmo}_0$ is proved to be the \emph{universal additive invariant} (see \cite{TA1Z}). Recall that an additive invariant is any functor $E\colon \textbf{dgcat}$ $\rightarrow \mathcal{D}$ taking values in an additive category $\mathcal{D}$ such that
\begin{itemize}
	\item[(\textbf{i})] it sends derived Morita equivalences to isomorphisms,\\
	
	\item[(\textbf{ii})] for any pre-triangulated dg category $\mathcal{A}$ admitting full pre-triangulated dg subcategories $\mathcal{B}$ and $\mathcal{C}$ such that $H^0(\mathcal{A})=\langle H^0(\mathcal{B}), H^0(\mathcal{C})\rangle$ is a semiorthogonal decomposition, the morphism $E(\mathcal{B})\oplus E(\mathcal{C})\rightarrow E(\mathcal{A})$ induced by the inclusions is an isomorphism.
\end{itemize}
Now let $G$ split simply connected semi-simple algebraic group over the field $k$ and $P$ a parabolic subgroup. We denote by $\widetilde{G}$ and $\widetilde{P}$ their universal covers. For the center $\widetilde{Z}\subset \widetilde{G}$ let $\mathrm{Ch}:=\mathrm{Hom}(\widetilde{Z},\mathbb{G}_m)$ be the character group. Furthermore, let $R(\widetilde{G})$ and $R(\widetilde{P})$ be the associated representation rings. Recall from \cite{PAZ},\S2 that there exits a finite free $\mathrm{Ch}$-homogeneous basis of $R(\widetilde{P})$ over $R(\widetilde{G})$. Moreover, to a 1-cocycle $\gamma\colon \mathrm{Gal}(k^s|k)\rightarrow G(k^s)$ one has the \emph{Tit's map} (see \cite{PAZ}, \S3 or \cite{KNUZ}, p.377) $\beta_{\gamma}\colon \mathrm{Ch}\rightarrow \mathrm{Br}(k)$ which is a group homomorphism and assigns to each character $\chi\in \mathrm{Ch}$ a central simple algebra $A_{\chi,\gamma}\in \mathrm{Br}(k)$, called \emph{Tit's algebra}. If $\rho_1,...,\rho_n$ is the $\mathrm{Ch}$-homogeneous $R(\widetilde{G})$ basis of $R(\widetilde{P})$ we write $\chi(i)$ for the character such that $\rho_i\in R^{\chi(i)}(\widetilde{P})$ (see \cite{PAZ}, \cite{KNUZ} and \cite{MPWZ} for details). Under this notation one has the following theorem:
\begin{thm}[\cite{TA1Z}, Theorem 2.1 (i)]
	Let $G$, $P$ and $\gamma$ be as above and $E\colon dgcat\rightarrow D$ an additive invariant. Then every $\mathrm{Ch}$-homogeneous basis $\rho_1,...,\rho_n$ of $R(\widetilde{P})$ over $R(\widetilde{G})$ give rise to an isomorphism
	\begin{eqnarray*}
		\bigoplus^n_{i=1}E(A_{\chi(i),\gamma})\stackrel{\sim}\longrightarrow E(\mathrm{perf}_{dg}({_\gamma}X)), 
	\end{eqnarray*}
	where $A_{\chi(i),\gamma}$ are the Tit's central simple algebras associated to $\rho_i$ via $\beta_{\gamma}\colon \mathrm{Ch}\rightarrow \mathrm{Br}(k)$.
\end{thm}
\begin{thm}[\cite{TA1Z}, Theorem 3.3]
	Let $G$, $P$ and $\gamma$ as in Theorem 4.1. Then $\bigoplus^n_{i=1}U(k)\simeq U(\mathrm{perf}_{dg}({_\gamma}X))$ if and only if the Brauer classes $[A_{\chi(i),\gamma}]$ are trivial.
\end{thm}

\section{Proof of Theorem 1.1}
\begin{proof}
	From the assumption that there is a semiorthogonal decomposition 
	\begin{eqnarray*}
		D^b(X)=\langle \pi^*D^b(Z)\otimes \mathcal{E}_1,...,\pi^*D^b(Z)\otimes \mathcal{E}_r\rangle 
	\end{eqnarray*} we obtain from \cite{KU1Z}, Proposition 4.10 that there are pretriangulated dg categories $\mathcal{T}_1,...,\mathcal{T}_r$ with $H^0(\mathcal{T}_i)=\pi^*D^b(Z)\otimes \mathcal{E}_i$. As $D^b(Z)$ is equivalent to $\pi^*D^b(Z)\otimes \mathcal{E}_i$ and since $D^b(Z)$ has a unique dg enhancement according to \cite{LUZ} we conclude
	\begin{eqnarray*}
		U(\mathrm{perf}_{dg}(X))\simeq U(\mathrm{perf}_{dg}(Z))\oplus...\oplus U(\mathrm{perf}_{dg}(Z)).
	\end{eqnarray*}
	According to Theorem 4.1 and \cite{KNUZ}, p.378, one has $U(\mathrm{perf}_{dg}(Z))=\bigoplus U(A_i)$ for some central simple $k$-algebras $A_i$. Assuming the existence of a full exceptional collection on $X$ we obtain
	\begin{eqnarray*}
		U(\mathrm{perf}_{dg}(X))&\simeq &U(k)\oplus...\oplus U(k)\\
		&\simeq & (\bigoplus_i U(A_i))\oplus...\oplus (\bigoplus_i U(A_i)).
	\end{eqnarray*}
	Then Theorem 4.2 implies that all $A_i$ must split. Now see \cite{NO2Z}, Proposition 5.8 to conclude that the 1-cocycle $\gamma$ defining the twisted flag $Z$ must be trivial too. This contradicts the assumption that $Z$ is a non-trivial twisted flag. 
\end{proof}

\begin{cor}
	Let $Z$ be as in Theorem 1.1 and $\mathcal{E}$ a vector bundle on $Z$. Then $\mathbb{P}_Z(\mathcal{E})$ cannot have a full exceptional collection.
\end{cor}
\begin{proof}
	Let $r$ be the rank of $\mathcal{E}$. Recall from \cite{OZ} that one has a semiorthogonal decomposition 
	\begin{eqnarray}
	D^b(\mathbb{P}_Z(\mathcal{E}))\simeq \langle \pi^*D^b(Z)\otimes \mathcal{O}_{\mathcal{E}},...,\pi^*D^b(Z)\otimes\mathcal{O}_{\mathcal{E}}(r-1) \rangle.
	\end{eqnarray}
	Note that this semiorthogonal decomposition also exists over arbitrary base fields $k$ (see for instance \cite{HUYZ}, p.184). It is easy to see that the triangulated category $D^b(Z)$ is equivalent to $\pi^* D^b(Z)\otimes\mathcal{O}_{\mathcal{E}}(i)$ via $\pi^*(-)\otimes\mathcal{O}_{\mathcal{E}}(i)$. Now Theorem 1.1 yields the assertion.
\end{proof}

\begin{cor}
	Let $Z$ be as in Theorem 1.1 and assume the base field $k$ is of characteristic zero. Furthermore, let $\mathcal{E}$ be a vector bundle on $Z$. Then $\mathrm{Grass}_Z(d,\mathcal{E})$ cannot have a full exceptional collection.
\end{cor}
\begin{proof}
	Let $r+1$ be the rank of $\mathcal{E}$ and denote by $\mathcal{R}$ the tautological subbundle of rank $d$ in $\pi^*(\mathcal{E})$. Moreover, let $P$ be the set of partitions $\lambda=(\lambda_1,...,\lambda_d)$ with $0\leq \lambda_d\leq...\leq \lambda_1\leq r+1-d$. One can choose a total order $\prec$ on $P$ such that $\lambda\prec \mu$ means that the Young diagram of $\lambda$ is not contained in that of $\mu$. 
	Recall from \cite{OZ} (alternatively see \cite{BAZ}) that one has a semiorthogonal decomposition
	\begin{eqnarray*}
		D^b(\mathrm{Grass}_Z(d,\mathcal{E}))=\langle ...\pi^*D^b(Z)\otimes \Sigma^{\lambda}(\mathcal{R}),...,\pi^*D^b(Z)\otimes \Sigma^{\mu}(\mathcal{R}),...\rangle
	\end{eqnarray*} where $\lambda\prec \mu$. 
	It is easy to see that $D^b(Z)$ is equivalent to $\pi^*D^b(Z)\otimes \Sigma^{\lambda}(\mathcal{R})$ via $\pi^*(-)\otimes\Sigma^{\lambda}(\mathcal{R})$. Indeed, this follows from adjunction of $\pi^*$ and $\pi_*$, projection formula (see \cite{HUYZ}) and the relative version of the Borel--Weil--Bott Theorem (see \cite{DLYZ}, Theorem 5.1). Then Theorem 1.1 yields the assertion.
\end{proof}
\section{Proof of Theorem 1.2}

\begin{lem}
	Let $\pi\colon X\rightarrow Z$ be a flat proper morphism between smooth projective $k$-varieties. Assume the existence of a semiorthogonal decomposition $D^b(X)=\langle \pi^*D^b(Z)\otimes \mathcal{E}_1,...,\pi^*D^b(Z)\otimes \mathcal{E}_r\rangle$ with $D^b(Z)$ being equivalent to $\pi^*D^b(Z)\otimes \mathcal{E}_i$ via $\pi^*(-)\otimes \mathcal{E}_i$. Assume furthermore that $\{\mathcal{A}_1,...,\mathcal{A}_m\}$ is a full exceptional collection for $Z$. Then the ordered set
	\begin{eqnarray*}
		S=\{\pi^*(\mathcal{A}_1)\otimes\mathcal{E}_1,...,\pi^*(\mathcal{A}_m)\otimes \mathcal{E}_1,...,\pi^*(\mathcal{A}_1)\otimes\mathcal{E}_r,...,\pi^*(\mathcal{A}_m)\otimes\mathcal{E}_r\}
	\end{eqnarray*}
	is a full exceptional collection for $X$.
\end{lem}
\begin{proof}
	Since $D^b(Z)$ is equivalent to $\pi^*D^b(Z)\otimes \mathcal{E}_i$ via the functor $\pi^*(-)\otimes \mathcal{E}_i$, we obtain $\mathrm{Hom}(\pi^*(\mathcal{A}_l)\otimes\mathcal{E}_i,\pi^*(\mathcal{A}_l)\otimes\mathcal{E}_i[p])\simeq \mathrm{Hom}(\mathcal{A}_l,\mathcal{A}_l[p])$. This implies that all the elements of $S$ are exceptional. From the assumption that $D^b(X)=\langle \pi^*D^b(Z)\otimes \mathcal{E}_1,...,\pi^*D^b(Z)\otimes \mathcal{E}_r\rangle$ is a semiorthogonal decomposition it is easy to conclude that $S$ generates $D^b(X)$ and that furthermore $\mathrm{Hom}(\pi^*(\mathcal{A}_l)\otimes\mathcal{E}_i,\pi^*(\mathcal{A}_q)\otimes\mathcal{E}_j[p])=0$ for all $p\in\mathbb{Z}$ whenever $i>j$.
\end{proof}
Denote by $\mathrm{NChow}(k)$ the category of noncommutative Chow motives (see \cite{TZ} for details). Now let $\mathrm{CSep}(k)$ be the full subcategory of $\mathrm{NChow}(k)$ consisting of objects of the form $U(A)$ with $A$ a commutative separable $k$-algebra. Analogously, $\mathrm{Sep}(k)$ denotes the full subcategory of $\mathrm{NChow}(k)$ consisting of objects $U(A)$ with $A$ a separable $k$-algebra. And finally, we write $\mathrm{CSA}(k)$ for the full subcategory of $\mathrm{Sep}(k)$ consisting of $U(A)$ with $A$ being a central simple $k$-algebra. 
\begin{proof} (of Theorem 1.2)\\
	Assume $\mathrm{rdim}(X)=0$. From \cite{ABZ}, Lemma 1.20 it follows 
	\begin{eqnarray*}
		D^b(X)=\langle D^b(K_1),...,D^b(K_s)\rangle, 
	\end{eqnarray*}
	where $K_1,...,K_r$ are \'etale $k$-algebras. From the assumption we then obtain
	\begin{eqnarray*}
		\langle \pi^*D^b(Z)\otimes \mathcal{E}_1,...,\pi^*D^b(Z)\otimes \mathcal{E}_r\rangle=\langle D^b(K_1),...,D^b(K_s)\rangle.
	\end{eqnarray*}
	Now \cite{KU1Z}, Proposition 4.10 provides us with pretriangulated dg categories $\mathcal{T}_1,...,\mathcal{T}_r$ such that $H^0(\mathcal{T}_i)\simeq \pi^*D^b(Z)\otimes \mathcal{E}_i$. As $D^b(Z)\simeq \pi^*D^b(Z)\otimes \mathcal{E}_i$ for $1\leq i\leq r$ and since the dg enhancement of $D^b(Z)$ is unique (see \cite{LUZ}) we get $U(\mathrm{perf}_{dg}(Z))\simeq U(\mathcal{T}_i)$ for all $1\leq i\leq r$. Applying noncommutative motives to the semiorthogonal decompositions gives
	\begin{eqnarray*}
		U(\mathrm{perf}_{dg}(Z))\oplus...\oplus U(\mathrm{perf}_{dg}(Z))\simeq U(K_1)\oplus...\oplus U(K_s).
	\end{eqnarray*}
	Assuming that $Z$ is either a Brauer--Severi variety over an arbitrary field $k$ or a generalized Brauer--Severi variety over a field of characteristic zero, we conclude from Theorem 4.1 that $U(\mathrm{perf}_{dg}(Z))= \bigoplus_j U(A_j)$ where $A_j$ are central simple $k$-algebras. Thus
	\begin{eqnarray*}
		\bigoplus_j U(A_j)^{\oplus r}\simeq U(K_1)\oplus...\oplus U(K_s).
	\end{eqnarray*}
	Recall from \cite{TAZ} that one has the following $2$-cartesian square of categories (see \cite{TAZ}, (2.16) and Corollary 2.13)
	\begin{displaymath}
	\begin{xy}
	\xymatrix{
		\{U(k)^{\oplus n}\mid n\geq 0\}\ar[r]^{} \ar[d]_{}    &   \mathrm{CSA}(k)^{\oplus}\ar[d]^{}                   \\
		\mathrm{CSep}(k) \ar[r]^{}             &   \mathrm{Sep}(k)             
	}
	\end{xy} 
	\end{displaymath}
	which gives an equivalence of categories $\{U(k)^{\oplus n}\mid n\geq 0\}\simeq\mathrm{CSA}(k)^{\oplus}\times_{\mathrm{Sep}(k)} \mathrm{CSep}(k)$. Here $\mathrm{CSA}(k)^{\oplus}$ denotes the closure of $\mathrm{CSA}(k)$ under finite direct sums.
	Now the above $2$-cartesian square, or more precise the universal property of fiber products, implies that such an isomorphism is possible if only if $K_1=...=K_s=k$ and all $A_j$ are split. On the other hand, if the 1-cocycle which determines $Z$ is trivial, $Z$ is a Grassmannian respectively a projective space. In both cases $Z$ admits a full exceptional collection and therefore $X$ has one. According to \cite{ABZ}, Lemma 1.20 one has $\mathrm{rdim}(X)=0$. Now assume the 1-cocycle which determines $Z$ is trivial. It is well known that $Z$ has a full exceptional collection. Then Lemma 6.1 provides us with a full exceptional collection for $X$. Again \cite{ABZ}, Lemma 1.20 immediately gives $\mathrm{rdim}(X)=0$.
	
	If $Z$ is a twisted quadric associated to $(A,\sigma)$ of degree $n$ with trivial discriminant
	we conclude from \cite{TA1Z}, Example 3.11 (see \cite{BLUZ} for a semiorthogonal decomposition of $D^b(Z)$) that the motive decomposes as 
	\begin{eqnarray*}
		U(\mathrm{perf}_{dg}(Z))\simeq \left(\bigoplus^{n-3}_{\substack{i\geq 0\\even}}U(k)\right)\oplus \left(\bigoplus^{n-3}_{\substack{i>0\\odd}}U(A)\right)\oplus U(C^{+}_0(A,\sigma))\oplus U(C^{-}_0(A,\sigma)).
	\end{eqnarray*}
	Here $k, A, C^{-}_0(A,\sigma)$ and $C^{+}_0(A,\sigma)$ are the minimal Tit's algebras of $Z$. In our case, the algebras $C^{-}_0(A,\sigma)$  and $C^{+}_0(A,\sigma)$ are central simple over $k$. Hence we get
	\begin{eqnarray*}
		U(\mathrm{perf}_{dg}(X))&\simeq&\Big(U(k)\oplus U(A)\oplus...\oplus U(C^{-}_0(A,\sigma))\oplus U(C^{+}_0(A,\sigma))\Big )^{\oplus r}\\
		&\simeq &U(K_1)\oplus...\oplus U(K_s).
	\end{eqnarray*}
	Note that this isomorphism follows also directly from \cite{TA1Z}, Example 3.11. Again the above $2$-cartesian square implies that $K_1=...=K_s=k$ and $A, C^{-}_0(A,\sigma)$ and $C^{+}_0(A,\sigma)$ are split. This implies that the 1-cocycle which determines $Z$ must be trivial.
	If $Z$ is twisted quadric associated to a central simple $\mathbb{R}$-algebra $(A,\sigma)$ with involution of orthogonal type, we refer to the proof of Theorem 5.6 in \cite{NO2Z} to conclude that $A$ must be split in this case.  
	Now assume the 1-cocycle which determines $Z$ is trivial. Then see \cite{KA2Z} or \cite{BLUZ} to conclude that $Z$ has a full exceptional collection. Then Lemma 6.1 provides us with a full exceptional collection for $X$. Again \cite{ABZ}, Lemma 1.20 immediately gives $\mathrm{rdim}(X)=0$.
\end{proof}
\begin{proof}(of Corollary 1.3)\\
	We assume $\mathrm{rdim}(X)=0$. Theorem 1.2 implies that the 1-cocycle which defines $Z$ must be trivial. Now see \cite{BLUZ} to conclude that $Z$ admits a full exceptional collection. Lemma 6.1 gives us a full exceptional collection for $X$. On the other hand, if $X$ admits a full exceptional collection then Lemma 1.20 of \cite{ABZ} immediately implies $\mathrm{rdim}(X)=0$.
\end{proof}
\begin{proof}(of Corollary 1.4)\\
	We prove the statement only for $\mathbb{P}_Z(\mathcal{E})$ as the other case can be shown in the same way. If $\mathbb{P}_Z(\mathcal{E})$ has a $k$-rational point, then so does $Z$. Now \cite{NO2Z}, Theorem 6.3, Proposition 6.8 and Proposition 6.10 imply that $Z$ admits a full exceptional collection. Lemma 6.1 provides us with a full exceptional collection for $X$ and Lemma 1.20 of \cite{ABZ} shows $\mathrm{rdim}(X)=0$.
	
	If $\mathrm{rdim}(X)=0$, Theorem 1.2 implies that the 1-cocycle defining $Z$ must be trivial. Hence $Z$ is a projective space or a smooth isotropic quadric and admits therefore a $k$-rational point $z_0\in Z$. Let $\pi^{-1}(z_0)\subset \mathbb{P}_Z(\mathcal{E})$ be the fiber. Note that $\pi^{-1}(z_0)\simeq \mathbb{P}_k^m$ where $m+1$ is the rank of $\mathcal{E}$. As $\mathbb{P}_k^m$ has a $k$-rational point, we also have one on $\mathbb{P}_Z(\mathcal{E})$. Now if $X$ is $k$-rational, $X$ admits a $k$-rational point. Therefore, $Z$ admits a $k$-rational point by Lang--Nishimura theorem. Now \cite{NO2Z}, Theorem 6.3, Proposition 6.8 and Proposition 6.10 imply that $Z$ admits a full exceptional collection. Lemma 6.1 provides us with a full exceptional collection for $X$ and Lemma 1.20 of \cite{ABZ} shows $\mathrm{rdim}(X)=0$. From the definition of $\mathrm{rcodim}(X)$ the assertion follows.
	This completes the proof.
\end{proof}

\begin{rema}
	\textnormal{The results in \cite{NO2Z} show that Theorems 1.1, 1.2 and thus Corollaries 1.3 and 1.4 also hold if $Z$ is the finite product of the considered varieties.}
\end{rema}
\section{Proof of Theorem 1.5}
\begin{proof}
	We have a semiorthogonal decomposition
	\begin{eqnarray*}
		D^b(X)=\langle \mathcal{A}_1,...,\mathcal{A}_r\rangle
	\end{eqnarray*}
	with $\mathcal{A}_i\simeq D^b(Z)$. For any of the varieties $Z$ from the assumption, we have $\mathrm{rdim}(Z)\leq \mathrm{ind}(A)-1$. Indeed, for Brauer--Severi varieties, see \cite{NO3Z}, Proposition 4.1, for generalized Brauer--Severi varieties see \cite{NO2Z}, Proposition 6.14 and for the twisted quadrics see \cite{NO2Z}, Theorem 6.16. Now the inequality $\mathrm{rdim}(Z)\leq \mathrm{ind}(A)-1$ and the fact that $D^b(Z)$ admits a semiorthogonal decomposition immediately implies
	\begin{eqnarray*}
		\mathrm{rdim}(X)\leq \mathrm{ind}(A)-1.
	\end{eqnarray*}
	For the second statement of our theorem, we first consider the split case $\mathrm{ind}(A)=1$. In this case, $D^b(Z)$ admits a full exceptional collection. Therefore, $D^b(X)$ admits a full exceptional collection (see \cite{OZ}). According to \cite{ABZ}, Lemma 1.20 we have $\mathrm{rdim}(X)=0$. Hence $\mathrm{rdim}(X)=\mathrm{ind}(A)-1$. Now let $\mathrm{ind}(A)=2$. Since $A$ is non-split in this case, we have $1\leq \mathrm{rdim}(X)$ according to Theorem 1.2 and $\mathrm{rdim}(X)\leq \mathrm{ind}(A)-1=1$ according to the first statement of the theorem. This implies $\mathrm{rdim}(X)=\mathrm{ind}(A)-1$. Finally, we consider the case $\mathrm{ind}(A)=3$. Again, since $A$ is non-split, we have $1\leq \mathrm{rdim}(X)$. We exclude the cases $\mathrm{rdim}(X)=0$ and $\mathrm{rdim}(X)=1$. From this, we get $\mathrm{rdim}(X)=2=\mathrm{ind}(A)-1$. Since $A$ is non-split, we immediately get $\mathrm{rdim}(X)\neq 0$. Assume $\mathrm{rdim}(X)=1$. From \cite{AB1Z}, Proposition 6.1.6 and 6.1.10 we conclude that if $\mathrm{rdim}(X)=1$ there must be a semiorthogonal decomposition of $D^b(X)$ whose components are either $D^b(K)$, where $K/k$ is a (finite) separable extension, or $D^b(D)$, where $D$ is a central division algebra with $\mathrm{ind}(D)\in\{1,2\}$, or $D^b(C)$, where $C$ is a smooth $k$-curve of positive genus. Note that $D^b(C)$ cannot be present because $K_0(X)$ is torsion free. Using the described semiorthogonal decomposition (that we get assuming $\mathrm{rdim}(X)=1$), we see that the noncommutative motive $U(\mathrm{perf}_{dg}(X))$ decomposes as
	\begin{eqnarray}
	U(\mathrm{perf}_{dg}(X))\simeq\bigoplus^n_{i=1}U(K_i)\oplus \left(\bigoplus^m_{j=1}U(D_j)\right)
	\end{eqnarray}
	for suitable central division algebras $D_j$ with $\mathrm{ind}(D_j)\in\{1,2\}$ and suitable separable extensions $K_i$.
	Note that $K_0(X)\simeq \mathbb{Z}^{\oplus h}$ and therefore $n+m=h$. After base change to $\bar{k}$ we also have $K_0(X_{\bar{k}})\simeq \mathbb{Z}^{\oplus h}$. 
	Now for any $D^b(K_i)$ we obtain after base change $D^b(K_i)_{\bar{k}}\simeq \prod^{r_i}_{q=1}D^b(\bar{k})$.
	Hence $\sum^n_{i=1}{r_i}+m=h$. But this implies $\sum^n_{i=1}{r_i}=n$ and therefore $r_i=1$. In the case where $Z$ is a (generalized) Brauer--Severi variety, the noncommutative motive $ U(\mathrm{perf}_{dg}(X))$ decomposes as 
	\begin{eqnarray}
	U(\mathrm{perf}_{dg}(X))\simeq\left(\bigoplus_{j}U(A^{\otimes q_j})\right)^{\oplus r}.
	\end{eqnarray}
	Note that there is a $j_0$ such that $q_{j_0}=1$. 
	The above isomorphisms (2) and (3) then give
	\begin{eqnarray*}
		\bigoplus^n_{i=1}U(k)\oplus \left(\bigoplus^m_{j=1}U(D_j)\right) \simeq \left(\bigoplus_{j}U(A^{\otimes q_j})\right)^{\oplus r}
	\end{eqnarray*} 
	Then by \cite{TAZ}, Theorem 2.19 we conclude that $A$ must be split or that $A$ must be Brauer-equivalent to $D_j$ for some $j$. This contradicts $\mathrm{ind}(A)=3$. 
	
	Now we consider the case where $Z$ is a twisted quadric. In this case the noncommutative motive decomposes as
	\begin{eqnarray*}
		U(\mathrm{perf}_{dg}(X))\simeq \left(\bigoplus_{s}U(k)\oplus \bigoplus_{t}U(A)\oplus U(C(A,\sigma))\right)^{\oplus r}, 
	\end{eqnarray*} 
	where $C(A,\sigma)$ denotes the Clifford algebra. Since we assumed $\mathrm{ind}(A)=3$, $C(A,\sigma)$ is the product of two central simple algebras $A_1$ and $A_2$. But then we have
	\begin{eqnarray*}
		\bigoplus^n_{i=1}U(k)\oplus \left(\bigoplus^m_{j=1}U(D_j)\right) \simeq \left(\bigoplus_{s}U(k)\oplus \bigoplus_{t}U(A)\oplus U(A_1)\oplus U(A_2))\right)^{\oplus r}
	\end{eqnarray*} 
	Then by \cite{TAZ}, Theorem 2.19 we conclude that $A$ must be split or that $A$ must be Brauer-equivalent to $D_j$ for some $j$. This contradicts $\mathrm{ind}(A)=3$. This completes the proof.
\end{proof}

{\small MATHEMATISCHES INSTITUT, HEINRICH--HEINE--UNIVERSIT\"AT 40225 D\"USSELDORF, GERMANY}\\
E-mail adress: novakovic@math.uni-duesseldorf.de

\end{document}